\documentclass[twoside,reqno,final]{amsart}

\usepackage{amsfonts,amsmath,amscd,amsthm,amssymb}
\usepackage{mathtools}
\usepackage[mathscr]{euscript} 
\usepackage{graphics}
\usepackage{wrapfig}
\usepackage{lscape}
\usepackage{rotating}
\usepackage{epsfig}
\usepackage{cite}
\usepackage{color}
\usepackage[notcite,notref]{showkeys}
\usepackage[vlined,ruled]{algorithm2e}
\usepackage{multirow}

\newtheorem{lemma}{Lemma}
\newtheorem{theorem}{Theorem}

\newtheorem{remark}{Remark}

\renewcommand{\b}[1]{{\boldsymbol{#1}}}

\providecommand{\NDOne}[1]{\mathbb{ND}^{I}_{#1}}

\newcommand{\NED}{N\'ed\'elec }

\def\restrict#1{\raise-0.2ex\hbox{\ensuremath|}_{#1}}

\newcommand{\VT}{\mathcal{V}}
\newcommand{\EG}{\mathcal{E}}
\newcommand{\FC}{\mathcal{F}}
\newcommand{\vt}{\mathsf{v}}
\newcommand{\eg}{\mathsf{e}}
\newcommand{\egI}{\mathsf{E}}
\newcommand{\fc}{\mathsf{f}}
\newcommand{\fcI}{\mathsf{F}}
\newcommand{\fcs}{{\mathsf{f}_s}}
\newcommand{\fct}{{\mathsf{f}_t}}

\newcommand{\bR}{\mathbb R}

\newcommand{\bld}[1]{\boldsymbol{#1}}
\newcommand{\curls}{{{\nabla\times}}}
\newcommand{\divs}{{\nabla\cdot}}
\newcommand{\grads}{{\nabla}}

\newtheorem{definition}[theorem]{Definition}

\newcommand{\pol}{\mathbb{P}}

\newcommand{\qol}{\mathbb{Q}}

\begin{document}
\title[Pyramidal sequence]{A lowest-order composite finite element exact sequence on pyramids}
\author{Mark Ainsworth}
\address{Division of Applied Mathematics, Brown University, 182 George St,
Providence RI 02912, USA.}
\email{Mark\_Ainsworth@brown.edu}
\thanks{First author gratefully acknowledges the partial support of this work
under AFOSR contract FA9550-12-1-0399.}

\author{Guosheng Fu}
\address{Division of Applied Mathematics, Brown University, 182 George St,
Providence RI 02912, USA.}
\email{Guosheng\_Fu@brown.edu}

\keywords{}
\subjclass{65N30. 65Y20. 65D17. 68U07}

\begin{abstract} 

Composite basis functions for pyramidal elements on the spaces $H^1(\Omega)$,
$H(\mathrm{curl},\Omega)$, $H(\mathrm{div},\Omega)$ and $L^2(\Omega)$ are
presented. In particular, we construct the lowest-order composite pyramidal
elements and show that they respect the de Rham diagram, i.e. we have an exact
sequence and satisfy the commuting property. Moreover, the finite elements are
fully compatible with the standard finite elements for the lowest-order
Raviart-Thomas-\NED sequence on tetrahedral and hexahedral elements. That is to
say, the new elements have the same degrees of freedom on the shared interface
with the neighbouring hexahedral or tetrahedra elements, and the basis
functions are conforming in the sense that they maintain the required level of
continuity (full, tangential component, normal component, ...) across the
interface. Furthermore, we study the approximation properties of the spaces as
an initial partition consisting of tetrahedra, hexahedra and pyramid elements
is successively subdivided and show that the spaces result in the same
(optimal) order of approximation in terms of the mesh size $h$ as one would
obtain using purely hexahedral or purely tetrahedral partitions. 
\end{abstract} \maketitle

\section{Introduction}

A key issue in performing finite element analysis in complex three dimensional
applications 
is the meshing of the domain $\Omega$. Given a choice, many practitioners would
opt for a mesh consisting entirely of brick, or hexahedral elements.
Unfortunately, meshing complicated domains using only hexahedral elements is
far from straightforward. By way of contrast, current mesh generators based on
tetrahedral elements are readily available and are routinely used to mesh even
quite complicated domains in three dimensions. 

The relative efficiency with which hexahedral elements can be used to fill
space compared with tetrahedral elements has led to an increasing
use~\cite{Baudouin2014,OwenSaigal01}  of \emph{tetrahedral-hexadral-pyramidal}
(THP) partitions in which hexahedral elements are used on part of the domain
whereas tetrahedral elements are used in sub-domains where the use of
hexahedral elements would pose difficulties. Pyramidal elements are used to
interface between hexahedral and tetrahedral elements.

At first glance, THP partitions would seem to provide the best of both
worlds---at least as far as the issue of meshing goes. Difficulties arise when
one turns to the question of choosing the basis functions for the pyramidal
elements. Here, we have in mind applications including: primal formulations of
elasticity, requiring elements in the space $H^1(\Omega)$ for which full
interelement continuity is needed; electromagnetics, requiring vector-valued
elements in the space $H(\mathrm{curl},\Omega)$ for which only the continuity
of tangential components is needed; mixed formulations of porous media flow,
using vector-valued elements in the space $H(\mathrm{div},\Omega)$ for which
only continuity of normal components is needed, along with scalar valued
elements in the space $L^2(\Omega)$. 

The choice of degrees of freedom and the types of basis function that can be
used in the pyramidal elements is constrained by the requirement for the
resulting elements to be compatible with neighbouring hexahedral and
tetrahedral finite elements. Indeed, the need for compatibility between
hexahedra-pyramids and between tetrahedra-pyramids dictates the choice of the
degree of freedom on the pyramid.  The problems start when one seeks to select
a basis for the pyramidal elements. 

Firstly, the basis functions for the hexadral and tetrahedral elements are
polynomial on the interfaces. This dictates that the basis functions on the
faces of the pyramids must be polynomials of exactly the same type as those on
the neighbouring element. This would be fine were it not for the fact that it
is \emph{impossible} to achieve this using polynomial basis functions on the
pyramidal elements. While it is quite possible to proceed by using
\emph{rational} basis functions on the pyramids
\cite{Bedrosian92,GradinaruHiptmair99,Sherwin97,SherwinWarburtonKarniadakis98,Warburton00,NigamPhillips12,NigamPhillips12b,BergotCohenDurufle10,BergotDurufle13a,BergotDurufle13b,BergotDurufle13c,FuentesKeithDemkowiczNagaraj15,ChanWarburton15,ChanWarburton16,Gillette16,CockburnFuCommuting},
this path gives rise to a new set of issues relating to computations involving
the rational functions and their approximation properties. 

Secondly, the pyramidal basis functions used to discretise the spaces
$H^1(\Omega)$, $H(\mathrm{curl},\Omega)$, $H(\mathrm{div},\Omega)$ and
$L^2(\Omega)$ should not be chosen in isolation. For instance, if finite
element discretisation of mixed formulations are to be stable and consistent
then it is necessary for the finite element spaces to be \emph{exact} and for
the associated de Rham diagram to \emph{commute}. These terms will be defined
more precisely later but for now it suffices to realise that there are
additional constraints beyond that of simply maintaining compatibility with
hexahedral and tetrahedral elements. 

An alternative approach avoiding the need for rational basis functions, and one
that is perhaps rather more akin to the spirit of the finite element method
itself, consists of using a \emph{composite} or \emph{macro} element basis for
the pyramidal elements. In other words, the basis functions on the individual
pyramids are, by analogy with the finite element method, constructed using
\emph{piecewise polynomials} on a subdivision of the pyramid. A natural choice
consists of subdividing the pyramid into a pair of tetrahedra
\cite{Wieners97,KnabnerSumm01,BluckWalker08,LiuDaviesYuanKvrivzek04,LiuDaviesKrivzekGuan11,AinsworthDavydovSchumaker16}. Of course, the use of composite 
element poses its own challenges.

In the present work we explore the use of composite basis functions for
pyramidal elements on the spaces $H^1(\Omega)$, $H(\mathrm{curl},\Omega)$,
$H(\mathrm{div},\Omega)$ and $L^2(\Omega)$. In particular, we construct the
lowest-order composite pyramidal elements and show that they respect the de
Rham diagram, i.e. we have an exact sequence and satisfy the commuting
property. Moreover, the finite elements are fully compatible with the standard
finite elements for the lowest-order {Raviart-Thomas-\NED}sequence
\cite{Nedelec80,ArnoldLogg14} on tetrahedral and hexahedral elements. That is
to say, the new elements have the same degrees of freedom on the shared
interface with the neighbouring hexahedral or tetrahedra elements, and the
basis functions are conforming in the sense that they maintain the required
level of continuity (full, tangential component, normal component, ...) across
the interface. Furthermore, we study the approximation properties of the spaces
as an initial THP partition is successively subdivided and show that the spaces
result in the same (optimal) order of approximation in terms of the mesh size
$h$ as one would obtain using purely hexahedral or purely tetrahedral
partitions. In short, all of the properties one would wish to have of
lowest-order pyramidal elements are shown to hold along with the added bonus
that composite elements are readily handled using standard finite element
technology. 

The existing literature \emph{loc. cit.} on composite elements is limited to
$H^1$-conforming finite elements---our finite element and basis functions
coincide with that proposed in \cite{Wieners97} for the $H^1(\Omega)$ case.
However, many problems of practical interest are naturally posed over a mixture
of the spaces $H(\mathrm{curl})$, $H(\mathrm{div})$ or $L^2$ in addition to
$H^1(\Omega)$.  As such the absence of lowest-order composite
$H(\mathrm{curl})$-, our $H(\mathrm{div})$-, and $L^2$-conforming finite
elements constitutes a severe limitation for the use of finite element
approximation on THP partitions. The present work fills this gap in the
literature. 

The remainder of this article is organized as follows. Section 2 gives a brief
overview of the standard lowest-order spaces on hexahedral and tetrahedral
elements with which we are seeking to maintain compatibility, and contains
preliminaries on the pyramidal geometry and polynomial spaces in two
dimensions.  In Section 3, we introduce the lowest-order composite finite
elements on a pyramid.  In Section 4, we present the lowest-order global finite
elements on tetrahedral-hexahedral-pyramidal ({\sf THP}) partitions and prove
their approximation properties on a sequence of uniformly refined (non-affine)
{\sf THP} partitions.

\section{Review of Lowest Order Hexahedral and Tetrahedral Elements}
\label{sec:notation} In this section, we introduce various notations and take
the opportunity to briefly review the lowest order finite elements on
hexahedral and tetrahedra with which we require the new pyramidal elements to
be compatible.  \subsection{Lowest-order finite elements on a tetrahedron and a
hexahedron} Let us first recall the lowest order finite elements on the
reference tetrahedron
\begin{subequations}
\label{reference}
\begin{align}
\label{tet}
 \widehat T = \{(x,y,z): 0<x,y,z<1, x+y+z <1\},
\end{align}
and the reference hexahedron 
\begin{align}
\label{cube}
\widehat H = \{(x,y,z): 0<x,y,z<1\}.
\end{align}
\end{subequations}

The lowest-order $H^1$-conforming finite elements and the associated (vertex-based) basis 
functions on the reference tetrahedron and reference hexahedron are given in Table \ref{table-H1}.
Here $\pol_k(\widehat T)$ stands for the space of polynomials of total degree no greater than $k$ on the 
reference tetrahedron, and
$\qol_k(\widehat H)$ stands for the space of tensor-product polynomials of degree no greater than $k$ in each argument on the reference hexahedron. 
\begin{table}[!ht]
\begin{tabular}{|c| c| c| c c|}
\hline
element $\widehat K$     & space $S(\widehat K)$ & $\#$ & \multicolumn{2}{|c|}{basis}\\
 \hline
\multirow{ 2}{*}{$\widehat T$} & \multirow{ 2}{*}{$\pol_1(\widehat T)$} & 
\multirow{2}{*}4 & \multicolumn{2}{|c|}{$\phi_1 = 1-x-y-z$, }\\
 & &  & \multicolumn{2}{|c|}{ $\phi_2 = x $, \;\; $\phi_3 = y,\;\; \phi_4 = z,$.}\\
 \hline
\multirow{ 4}{*}{$\widehat H$} & \multirow{ 4}{*}{$\qol_1(\widehat H)$} & 
\multirow{4}{*}8 & 
 \multicolumn{2}{|c|}{$\phi_1 = (1-x)(1-y)(1-z), $  $\phi_2 = x(1-y)(1-z),$}\\
 & &  & 
 \multicolumn{2}{|c|}{$\phi_3 = xy(1-z), $  $\phi_4 = (1-x)y(1-z)$,}\\ 
 & &  & 
  \multicolumn{2}{|c|}{$\phi_5 = (1-x)(1-y)z, $  $\phi_6 = x(1-y)z$,}\\ 
 & &  & 
  \multicolumn{2}{|c|}{$\phi_7 = xyz, $  $\phi_8 = (1-x)yz$.}\\ 
 \hline
\end{tabular}
\caption{Lowest order $H^1$-conforming finite elements on the reference tetrahedron and 
reference hexahedron.}
 \label{table-H1}
\end{table}

The lowest-order $H(\mathrm{curl})$-conforming finite elements and the associated (edge-based) basis 
functions on the reference tetrahedron and reference hexahedron are given in Table \ref{table-Hcurl}.
Here 
the lowest-order \NED spaces of the first kind are given as follows:
\begin{align*}
\mathbb{ND}_0(\widehat T) :=&\;\pol_0(\widehat T)^3\oplus \bld x\times \pol_0(\widehat T)^3,\\
 \mathbb{ND}_0(\widehat H) :=&\;
\left(\pol_{0,1,1}(\widehat H),\pol_{1,0,1}(\widehat H),\pol_{1,1,0}(\widehat H)\right)^{\sf{T}},
\end{align*}
where $\pol_{\ell,m,n}(\widehat H)$ stands for the space of tensor-product polynomials of 
degree no greater than $\ell$ in the first argument, 
$m$ in the second argument, and $n$ in the third argument on the reference hexahedron. 
\begin{table}[!ht]
\resizebox{0.9\columnwidth}{!}{%
\begin{tabular}{|c| c| c| c c|}
\hline
element $\widehat K$     & space $\bld E(\widehat K)$ & $\#$ & \multicolumn{2}{|c|}{basis}\\
 \hline
\multirow{ 3}{*}{$\widehat T$} & \multirow{ 3}{*}{
$\mathbb{ND}_0(\widehat T)$
} & 
\multirow{3}{*}6 & 
\multicolumn{2}{|c|}{$\bld\varphi_1 = x\grads y-y\grads x $, $\bld\varphi_2 = z\grads x+z\grads y+(1-x-y)\grads z$,}
\\
 & &  & 
 \multicolumn{2}{|c|}{$\bld\varphi_3 = y\grads z-z\grads y $, 
 $\bld\varphi_4 = (1-y-z)\grads x+x\grads y+x\grads z,$}\\
 & &  & 
 \multicolumn{2}{|c|}{$\bld\varphi_5 = z\grads x-x\grads z $, 
 $\bld\varphi_6 = y\grads x+(1-x-z)\grads y+y\grads z.$}\\
\hline
\multirow{ 4}{*}{$\widehat H$} & \multirow{ 4}{*}{$\mathbb{ND}_0(\widehat H)$} & 
\multirow{4}{*}{12} & 
 \multicolumn{2}{|c|}{$\bld\varphi_1 = (1-y)(1-z)\grads x,\bld\varphi_2 = (1-y)z\grads x,
\bld\varphi_3= y(1-z)\grads x
 $,}\\
 & &  & 
 \multicolumn{2}{|c|}{$\bld\varphi_4 = (1-z)(1-x)\grads y,\bld\varphi_5 = (1-z)x\grads y,
\bld\varphi_6= z(1-x)\grads y
 $,}\\
 & &  & 
 \multicolumn{2}{|c|}{$\bld\varphi_7 = (1-x)(1-y)\grads z,\bld\varphi_8 = (1-x)y\grads z,
\bld\varphi_9= x(1-y)\grads z
 $,}\\
 & &  & 
 \multicolumn{2}{|c|}{$\bld\varphi_{10} = yz\grads x,\bld\varphi_{11} = zx\grads y,
\bld\varphi_{12}= xy\grads z
 $.}\\
 \hline
\end{tabular}
}

\caption{Lowest order $H(\mathrm{curl})$-conforming finite elements on the reference tetrahedron and 
reference hexahedron.
}
 \label{table-Hcurl}
\end{table}

The lowest-order $H(\mathrm{div})$-conforming finite elements and the associated (face-based) basis 
functions on the reference tetrahedron and reference hexahedron are given in Table \ref{table-Hdiv}.
Here 
the lowest-order Raviart-Thomas spaces are given as follows:
\begin{align*}
\mathbb{RT}_0(\widehat T) :=&\;\pol_0(\widehat T)^3\oplus \bld x\, \pol_0(\widehat T),\\
 \mathbb{RT}_0(\widehat H) :=&\;
\left(\pol_{1,0,0}(\widehat H),\pol_{0,1,0}(\widehat H),\pol_{0,0,1}(\widehat H)\right)^{\sf{T}}.
\end{align*}
\begin{table}[!ht]
\begin{tabular}{|c| c| c| c c|}
\hline
element $\widehat K$     & space $\bld V(\widehat K)$  & $\#$ & \multicolumn{2}{|c|}{basis}\\
 \hline
\multirow{ 4}{*}{$\widehat T$} & \multirow{ 4}{*}{
$\mathbb{RT}_0(\widehat T)$
} & 
\multirow{4}{*}4 & 
\multicolumn{2}{|c|}{$\bld\psi_1 = 2(x\grads x+y\grads y+z\grads z)$,}\\
& & & \multicolumn{2}{|c|}{
$\bld\psi_2 = 2((x-1)\grads x+y\grads y+z\grads z)$,}
\\
 & &  & 
 \multicolumn{2}{|c|}{$\bld\psi_3 = 2(x\grads x+(y-1)\grads y+z\grads z)$,}\\
&&&  \multicolumn{2}{|c|}{ 
$\bld\psi_4 = 2(x\grads x+y\grads y+(z-1)\grads z)$.}\\
\hline
\multirow{3}{*}{$\widehat H$} & \multirow{3}{*}{$\mathbb{RT}_0(\widehat H)$} & 
\multirow{3}{*}{6} & 
 \multicolumn{2}{|c|}{$\bld\psi_1 = (x-1)\grads x,\;\;\;\bld\psi_2 = x\grads x,$}\\
 &&&
 \multicolumn{2}{|c|}{$\bld\psi_3= (y-1)\grads y,\;\;\; \bld\psi_4= y\grads y$,} \\
 & &  & 
 \multicolumn{2}{|c|}{$\bld\psi_5= (z-1)\grads z,\;\;\; \bld\psi_6= z\grads z$.} \\
 \hline
\end{tabular}
\caption{Lowest order $H(\mathrm{div})$-conforming finite elements on the reference tetrahedron and 
reference hexahedron.
}
 \label{table-Hdiv}
\end{table}

The lowest-order $L^2$-conforming finite elements and the associated (cell-based) basis 
functions on the reference tetrahedron and reference hexahedron are given in Table \ref{table-L2}.
\begin{table}[!ht]
\begin{tabular}{|c| c| c| c c|}
\hline
element $\widehat K$     & space $W(\widehat K)$  & $\#$ & \multicolumn{2}{|c|}{basis}\\
 \hline
\multirow{ 1}{*}{$\widehat T$} & \multirow{1}{*}{
$\pol_0(\widehat T)$
} & 
\multirow{1}{*}1 & 
\multicolumn{2}{|c|}{$\psi_1 = 6$}\\
\hline
\multirow{1}{*}{$\widehat H$} & \multirow{1}{*}{$\qol_0(\widehat H)$} & 
\multirow{1}{*}{1} & 
 \multicolumn{2}{|c|}{$\psi_1 = 1$}\\
 \hline
\end{tabular}
\caption{Lowest order $L^2$-conforming finite elements on the reference tetrahedron and 
reference hexahedron.
}
 \label{table-L2}
\end{table}

Let $\Phi_K:\widehat K\rightarrow K$ be a continuously differentiable, invertible and surjective map. 
If $\widehat{\bld {x}}$ denotes a coordinate system on the reference element $\widehat K$, then 
$\bld x = \Phi_K(\widehat{\bld {x}})$ is the corresponding coordinate system on the physical element $K$.
The
Jacobian matrix of the mapping $\Phi_K$ with respect to the reference coordinates and its determinant 
are denoted by
\[
 F_K(\widehat{\b x}) = \left(\frac{\partial\Phi_{K,i}}{\partial \hat x_j}(\widehat {\b x})\right)_{1\le i,j\le 3},\quad
 J_K(\widehat{\b x}) = \det (F_K(\widehat{\b x})).
\]
Finite elements on a {\it physical} element $K$ are defined in terms of  a reference element $\widehat K$ 
via mapping in the usual way \cite{Monk03}:
\begin{subequations}
\label{space-mapping}
\begin{alignat}{4}
\label{space-mapping-1}
 S(K) :=&\; \{ v\in H^1(K):\;
&& \;\; v = \widehat v\circ \Phi_K^{-1}, \quad &&\forall \;\widehat v\in  S(\widehat{K})
 \}, \\
\label{space-mapping-2}
 \b E(K) :=&\; \{\b v\in H(\mathrm{curl},K):\;
&&\;\; \b  v = F_K^{-T}\widehat {\b v}\circ \Phi_K^{-1}, \quad &&\forall\; \widehat{\b v}\in  \b E(\widehat{K})
 \}, \\
\label{space-mapping-3}
\b V(K) :=&\; \{\b v\in H(\mathrm{div}, K):\;
 &&\;\;\b v = J_K^{-1}F_K \widehat {\b v}\circ \Phi_K^{-1}, \quad&&\forall\;  \widehat{\b v}\in \b V(\widehat{K})
 \}, \\
\label{space-mapping-4}
W(K) :=&\; \{ v\in L^2(K):\;
 &&\;\; v = J_K^{-1} \widehat v\circ \Phi_K^{-1}, \quad&&\forall\;  \widehat v\in  W(\widehat{K})
 \},
\end{alignat}
\end{subequations}
where $\widehat K$ is either the reference tetrahedron $\widehat T$, or the 
reference hexahedron $\widehat H$, and
the reference finite elements $S(\widehat{K})$, 
$\b E(\widehat{K})$, $\b V(\widehat{K})$, and $W(\widehat{K})$
are given in Table \ref{table-H1}--\ref{table-L2}.
In particular, 
a basis for the physical element is obtained by applying the appropriate mapping to the basis defined on the reference element.

\subsection{Regular pyramid}
Let us now introduce notation on the pyramidal element that will be used later on. 
Let  $K\subset \bR^3$ be a {\it regular pyramid}
consisting of five vertices $\vt_1,\cdots,\vt_5$ ordered as shown in Figure \ref{fig-pyramid}, 
where $\vt_1,\vt_2, \vt_3,\vt_4$ form a parallelogram.
The pyramid is cut into two tetrahedra, $T_1$ and $T_2$,  by forming a new triangular face using 
 the vertices $\vt_1$, $\vt_3$, and $\vt_5$ as shown in Figure \ref{fig-pyramid}. 
The local vertex ordering of the two tetrahedra is also depicted in Figure \ref{fig-pyramid}.
\begin{figure}[!ht]
\centerline{
\psfig{file=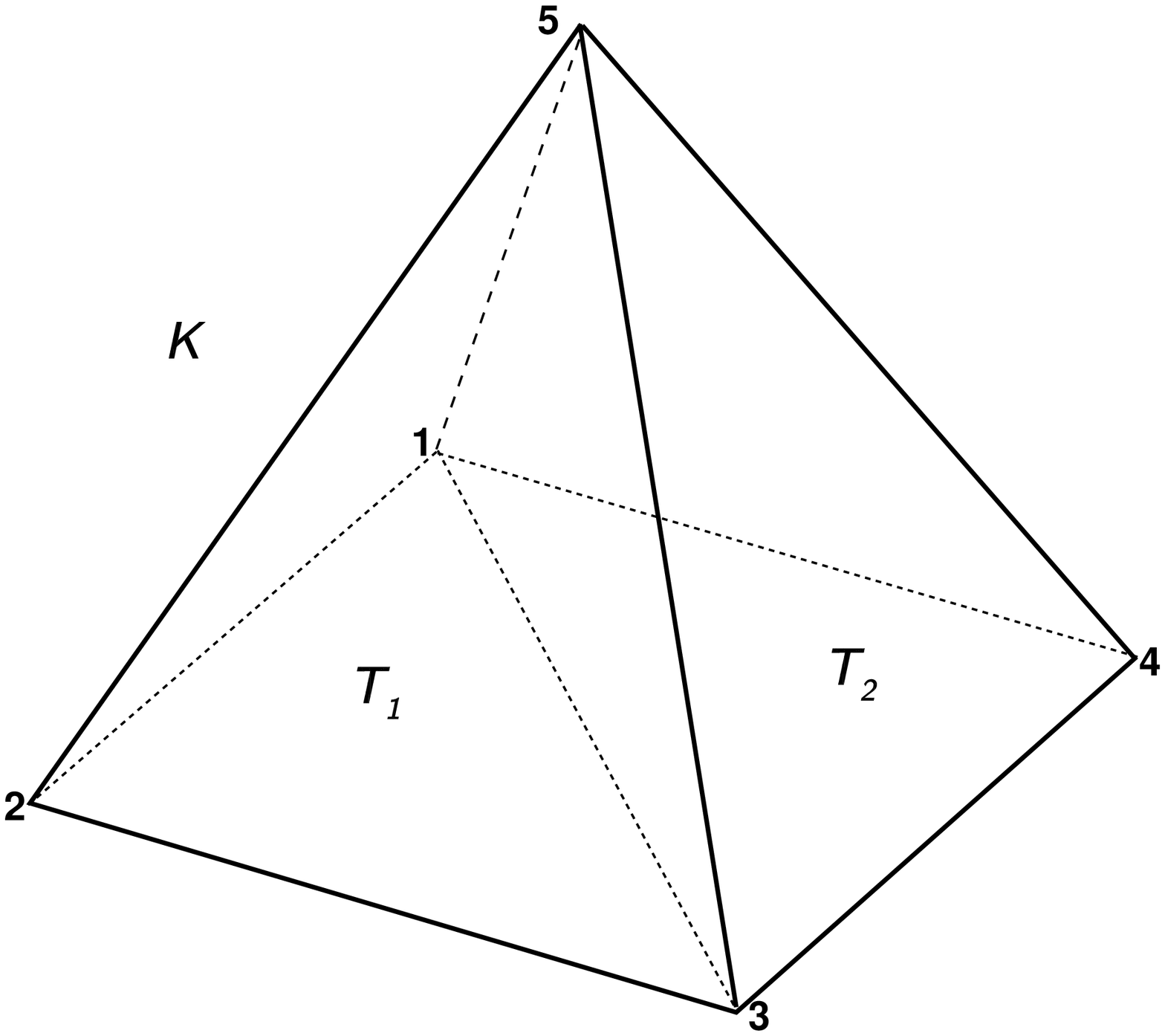,width=2.5in}
\psfig{file=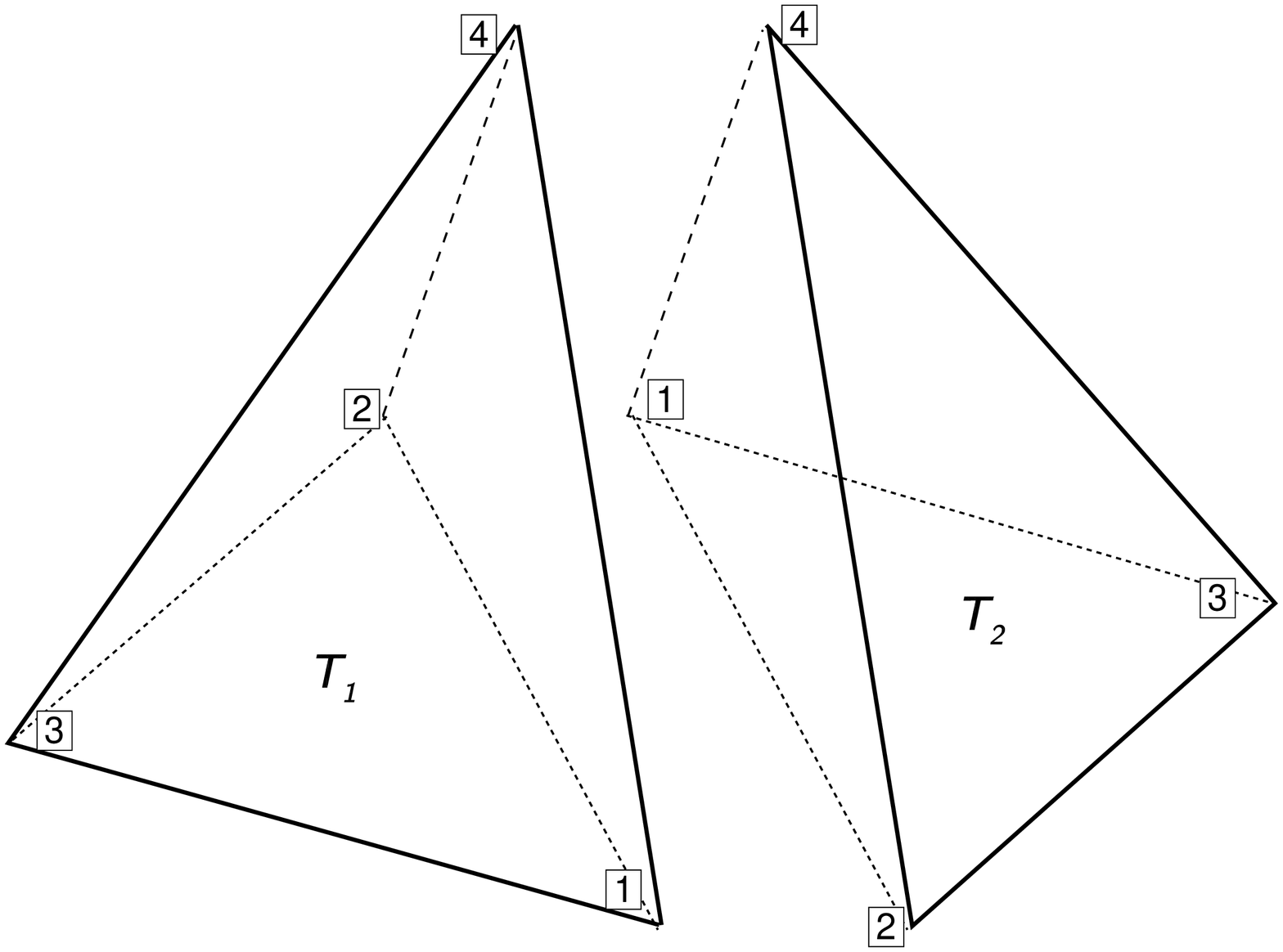,width=2.5in}}
\caption{Left: a regular pyramid cut into two tetrahedra. 
Right: local vertex ordering of each tetrahedra.
} \label{fig-pyramid}
\end{figure}

The pyramid consists of five vertices \[\VT(K) = \{\vt_1,\vt_2,\vt_3,\vt_4,\vt_5\},\] 
eight  edges 
 \[\EG(K)=\{\eg_{12}, \;\eg_{14}, \;\eg_{15},\;\eg_{23},\;\eg_{25},\;\eg_{34},\;\eg_{35},\;\eg_{45}
 \},\]
where $\eg_{ij}=[\vt_i,\vt_j]$ is the edge connecting vertices $\vt_i$ and $\vt_j$, 
and five faces 
\[\FC(K)=\{\fc_{125},\;\fc_{145},\;\fc_{235},\;\fc_{345},\;\fc_{1234}
\},
\]
where $\fc_{ijk} = [\vt_i,\vt_j,\vt_k]$ is a triangular face, and 
$\fc_{1234} = [\vt_1, \vt_2,\vt_3,\vt_4]$ is the base parallelogram face.
Moreover, let
$
\FC_t(K)\subset  \FC(K)
$
denote the set of triangular faces of $K$, and 
$
\FC_s(K)\subset  \FC(K)
$
denote the set containing the base parallelogram face of $K$.

\subsection{Polynomial spaces in two dimensions}
Given an element $T\subset \bR^2$, we us
$\pol_p(T)$ to denote the space of polynomials 
with total degree at most $p$ defined on $T$.
When $T\subset \bR^2$ is the unit square, we use 
$\pol_{p_1,p_2}(T)$ to denote the space of tensor-product polynomials with degree at most 
$p_1$ in the $x$-coordinate, and $p_2$ in the $y$-coordinate. 
We also denote $\qol_{p}(T) : = \pol_{p,p}(T)$.

The lowest-order \NED spaces of the first kind (rotated Raviart-Thomas) 
on the reference triangle $\hat\fc_t$ and reference square $\hat\fc_s$ in $\bR^2$
with Cartesian coordinates $(x,y)$ are given as follows:
\begin{align*}
 \NDOne{0}(\hat\fc_t) : = &\; \pol_0(\hat\fc_t)^2 \oplus \bld x \times \pol_0(\hat\fc_t),
 \\
 \NDOne{0}(\hat\fc_s) : = &\;
 [\pol_{0,1}(\hat\fc_s),\; \pol_{1,0}(\hat\fc_s)]^T.
\end{align*}
Here $\bld x\times v$ denotes the vector $[y\, v, -x\, v]^T$ where $v$ is  a scalar-valued function.
The lowest-order \NED space of the first kind on a physical  triangle or quadrilateral in $\bR^3$ is defined 
in terms of the corresponding space on the reference element via the standard covariant Piola mapping:
\begin{subequations}
\label{NED2D}
\begin{align}
\label{NED2D-1}
 \NDOne{0}(\fct) : = &\; 
 \{(J_t^\dagger)^{T} \bld\phi\,\circ F_t^{-1}:\;\;\;
 \bld\phi \in \NDOne{0}(\hat\fc_t)
 \},
 \\
 \NDOne{0}(\fcs) : = &\; 
 \{(J_s^\dagger)^{T} \bld\phi\,\circ F_s^{-1}:\;\;\;
 \bld\phi \in \NDOne{0}(\hat\fc_s),
 \}
\end{align}
\end{subequations}
where $F_t$ is the (affine) mapping from the reference triangle $\hat\fc_t$ in $\bR^2$ to the physical triangle $\fc_t$ in $\bR^3$,
$J_t\in \bR^{3\times 2}$ is the Jacobian matrix, and 
\[
 J_t^{\dagger} = ( J_t^T\,J_t)^{-1} J_t^T
\]
is the pseudo-inverse. Similarly, 
$F_s$ is the bilinear mapping from $\hat\fc_s$ to $\fc_s$,
and $J_s\in \bR^{3\times 2}$ is the associated Jacobian matrix.

Throughout this paper, 
we write $a\preceq b$, when
$a\le C\, b$ with a generic constant $C$ independent of $a, b$, and the mesh size.

\section{Lowest-order composite  finite elements on a pyramid}
In this section, we present the lowest-order composite  finite element family on a regular pyramid.
These finite elements have degrees of freedom compatible with standard finite elements on tetrahedra and hexahedra. 
Moreover, the set functions used to define element reduce to those used for tetrahedra and hexahedra on the faces. They
can be naturally combined with 
the standard, lowest-order, {Raviart-Thomas-\NED}finite elements to provide 
$H^1$-, $H(\mathrm{curl})$-, $H(\mathrm{div})$-, and $L^2$-conforming finite element spaces on a conforming hybrid mesh consisting of 
tetrahedra, pyramids, and hexahedra. 
Commuting diagrams are also respected at both element level on a single pyramid and globally on a hybrid mesh.

The classic definition of a finite element \cite{BrennerScott08,Ciarlet78}
as a triple is recalled below for convenience:
\begin{definition}
A finite element consists of a triple 
$(K, \mathcal{P}_K, \Sigma_K)$ where:
\begin{enumerate}
 \item the {\em \textbf{element domain}} $K\subset \bR^n$  is a bounded closed set with nonempty interior and piece-wise smooth boundary,
 \item the space of {\em\textbf{shape functions}} $\mathcal{P}_K$ 
 is a finite-dimensional space of functions on $K$,
 \item the set of {\em \textbf{nodal variables}} $\Sigma_K=\{N_1^K,\cdots, N_k^K\}$, or {\em \textbf{degrees of freedom}},  forms a basis for $\mathcal{P}_K'$, where $\mathcal{P}_K'$ denotes the dual space of $\mathcal{P}_K$.
\end{enumerate}
\end{definition}

We shall specify our spaces in terms of the {\it nodal basis}
$\{\phi_1,\phi_2,\cdots, \phi_k\}$ for $\mathcal{P}_K$ 
dual to $\Sigma_K$ (i.e. $N_i(\phi_j)= \delta_{ij}$).

To begin, 
recall that 
the regular pyramid $K$ is split into two tetrahedra $T_1$ and $T_2$  by connecting vertices $\vt_1$ and $\vt_3$, as shown in Figure \ref{fig-pyramid}. 
The barycentric coordinates on each tetrahedron will be used to construct our basis functions.
Let $\lambda_i^{T_\ell}$ ($i=1,2,3,4,\; \ell = 1,2$) denote the affine function that vanishes at 
all vertices of the tetrahedron $T_\ell$ except the $i$-th vertex, where it attains value $1$.
For example, with vertices ordered as in Figure \ref{fig-pyramid}, 
$\lambda_1^{T_1}$ is the affine function, which vanishes on the face $\fc_{125}$ and takes the value $1$ at 
the vertex $\vt_3$.

Next, we present a new, lowest order, finite element sequence on the pyramid which we view as a composite
element formed from the tetrahedra $T_1$ and $T_2$. We then show that the elements satisfy the  exact sequence 
property and exhibit explicit interpolation
operators which are shown to satisfy the commuting diagram property. 

\subsection{The $H^1$-conforming finite element}
We present the lowest-order $H^1$-conforming composite finite element on 
a regular pyramid.
\begin{definition}
\label{H1-def}
The lowest-order 
composite $H^1$-conforming finite element is is defined by 
$(K, S(K), \Sigma^S_K)$ where
\begin{itemize}
\item $K$ is a regular pyramid with parallelogram base,
 \item
 \begin{alignat}{3}
\label{H1}
S(K) :=&\; \{v \in H^1(K):&& \;
v\restrict{T_\ell}\in \pol_2(T_\ell),\;\; \ell = 1,2,\\
& &&\;\;\; v\restrict{\fc}\in \pol_1(\fc),\quad\forall \fc\in \FC_t(K),\nonumber\\
& &&\;\;\; v\restrict{\fc}\in \qol_{1}(\fc),\quad\forall \fc\in \FC_s(K).\} \nonumber
\end{alignat}
\item $\Sigma_K^S = \mathrm{span}\{N_i^S:\; i=1,2,3,4,5\}$ where 
\[
N_i^S:= N_i^{\vt}: v\longrightarrow v(\vt_i)\quad \forall \vt_i\in \VT_K.
\]
are point evaluation functionals at the vertices of $K$.
\end{itemize}
\end{definition}

\begin{lemma}
\label{lemma-1}
The triplet  $(K,S(K), \Sigma^{S}_K)$ is a finite element of dimension 5.
The set $\{\phi^S_i:\;i=1,2,\cdots,5\}$ is the nodal basis dual to 
$\Sigma^S_K$, where 
 \begin{alignat*}{3}
  \phi_1^S =&\;
  \left\{ \begin{tabular}{l l}
            $\lambda_2^{T_1}-\lambda_1^{T_1}\lambda_2^{T_1}$ & on $T_1$\vspace{0.2cm}\\
            $\lambda_1^{T_2}-\lambda_1^{T_2}\lambda_2^{T_2}$ & on $T_2$
            \end{tabular}\right.,
\quad &&  
  \phi_2^S =\left\{ \begin{tabular}{l l}
            $\lambda_3^{T_1}+\lambda_1^{T_1}\lambda_2^{T_1}$ & on $T_1$\vspace{0.2cm}\\
            $\lambda_1^{T_2}\lambda_2^{T_2}$ & on $T_2$
            \end{tabular}\right.,\\
\phi_3^S =&
\left\{ \begin{tabular}{l l}
            $\lambda_1^{T_1}-\lambda_1^{T_1}\lambda_2^{T_1}$ & on $T_1$\vspace{0.2cm}\\
            $\lambda_2^{T_2}-\lambda_1^{T_2}\lambda_2^{T_2}$ & on $T_2$
            \end{tabular}\right.,
\quad &&  
  \phi_4^S =
  \left\{ \begin{tabular}{l l}
            $\lambda_1^{T_1}\lambda_2^{T_1}$ & on $T_1$\vspace{0.2cm}\\
            $\lambda_3^{T_2}+\lambda_1^{T_2}\lambda_2^{T_2}$ & on $T_2$
            \end{tabular}\right.,
\\
  \phi_5^S =&
   \left\{ \begin{tabular}{l l}
            $\lambda_4^{T_1}$ & on $T_1$\vspace{0.2cm}\\
            $\lambda_4^{T_2}$ & on $T_2$
            \end{tabular}\right.,
 \end{alignat*}
 i.e., $
N_i^S (\phi_j^S) = \delta_{ij}. 
$
\end{lemma}
\begin{proof}
First, it is easy to check that each of the above functions $\phi_i^S$ belongs to $S(K)$ and satisfies 
$N_i^S(\phi_j^S)=\delta_{ij}$. Hence, 
\begin{align}
\label{h1x}
\dim S(K)\ge 
\dim \left(\mathrm{span}\{\phi_i^S:\;1\le i\le 5\}\right) = 5.
\end{align}

Now, we prove unisolvence of the set of degrees of freedom $\Sigma^S_K$.
Let $u\in S(K)$ be such that $N_i^S (u) = 0$ for $i= 1,2,\cdots,5$.
In other words,  $u$ vanishes at all vertices of $K$. 
Since the restriction of $u$ to a triangular face is a linear polynomial or to the quadrilateral face is a bilinear polynomial, 
we immediately have 
$u\restrict{\fc} = 0$ on all faces $\fc\in\FC(K)$. 
Hence, $u\restrict{T_\ell}\in \pol_2(T_\ell)$, for $\ell=1,2$, and $u$
vanishes on three faces of $T_\ell$.
This implies that $u=0$. Hence, we have 
$\dim S(K) \le \dim \Sigma^S_K\le 5$.
Inequality \eqref{h1x} now shows 
$\dim S(K) = 5$, and the unisolvence of the set $\Sigma^S_K$ follows.
\end{proof}

\subsection{The $H(\mathrm{curl})$-conforming finite element}
We present the lowest-order $H(\mathrm{curl})$-conforming composite finite element on 
a regular pyramid.
To simplify notation, we label the edges as follows:
\begin{align*}
\egI_{1} = \eg_{12}, \;
\egI_{2} = \eg_{14},\;
\egI_{3} = \eg_{15},\;
\egI_{4} = \eg_{23},\\
\egI_{5} = \eg_{25},\;
\egI_{6} = \eg_{34},\;
\egI_{7} = \eg_{35},\;
\egI_{8} = \eg_{45},
\end{align*}
and orient the edges so that the the tangential vector $\b t_i$ on $\egI_{i}$ point from the 
 lower vertex number to higher vertex number.

\begin{definition}
\label{Hcurl-def}
The lowest-order 
composite $H(\mathrm{curl})$-conforming finite element is is defined by 
$(K, \b E(K), \Sigma^{\b E}_K)$ where
\begin{itemize}
\item $K$ is a regular pyramid with parallelogram base,
 \item 
 \begin{alignat}{3}
\label{Hcurl}
\bld E(K) :=&\; \{\b v \in H(\mathrm{curl},K):&& \;
\b v\restrict{T_\ell}\in \pol_1(T_\ell)^3,\;\; \ell = 1,2,\\
& &&\!\!\!\!\!\!\!\!\!\!\!\!\!\!\!\!\!\!\b n\times (\b v\times \b n)\restrict{\fc}\in \NDOne{0}(\fc),\quad\forall\, \fc\in \FC(K)\},\nonumber
\end{alignat}
where $\b n$ is the unit outward normal vector on face $\fc$.
\item $\Sigma_K^{\b E}=\mathrm{span}\{N_i^{\b E}:\;1\le i\le 8\}$, where
\[
 N_i^{\b E}:= N_i^{\egI}: \b v\longrightarrow \int_{\egI_i} \b v\cdot \b t_i\,\mathrm{ds} \quad \forall\, \egI_i\in \EG_K.
\]
\end{itemize}
\end{definition}

We denote the {\it Whitney edge form} on the tetrahedral $T_\ell$, $\ell=1,2$, as
\[
\b\omega_{ij}^{T_\ell}: = \lambda_i^{T_\ell}\grads\lambda_j^{T_{\ell}}
-\lambda_j^{T_\ell}\grads\lambda_i^{T_{\ell}}.
\]

A basis for the space $\b E(K)$ is given as follows.
\begin{lemma}
\label{lemma-2}
The triplet  $(K,\b E(K), \Sigma^{\b E}_K)$ is a finite element of dimension 8.
The set $\{\b\varphi^{\b E}_i:\;i=1,2,\cdots,8\}$ is the nodal basis dual to 
$\Sigma^{\b E}_K$, where 
 \begin{alignat*}{3}
 \b\varphi^{\b E}_1 =&
\left\{ \begin{tabular}{l l}
            $\b\omega_{23}^{T_1}+\lambda_2^{T_1}\grads\lambda_1^{T_1}$ & on $T_1$\vspace{0.2cm}\\
            $\lambda_1^{T_2}\grads\lambda_2^{T_2}$ & on $T_2$
            \end{tabular}\right.,
\quad && 
  \b\varphi^{\b E}_2 =
  \left\{ \begin{tabular}{l l}
            $\lambda_2^{T_1}\grads\lambda_1^{T_1}$ & on $T_1$\vspace{0.2cm}\\
            $\b\omega_{13}^{T_2}+\lambda_1^{T_2}\grads\lambda_2^{T_2}$ & on $T_2$
            \end{tabular}\right.,
\\
  \b\varphi^{\b E}_3 =&\left\{ \begin{tabular}{l l}
            $\b\omega_{24}^{T_1}$ & on $T_1$\vspace{0.2cm}\\
            $\b\omega_{14}^{T_2}$ & on $T_2$
            \end{tabular}\right.,
\quad &&  
\b\varphi^{\b E}_4 =
  \left\{ \begin{tabular}{l l}
            $\b\omega_{31}^{T_1}-\lambda_1^{T_1}\grads\lambda_2^{T_1}$ & on $T_1$\vspace{0.2cm}\\
            $-\lambda_2^{T_2}\grads\lambda_1^{T_2}$ & on $T_2$
            \end{tabular}\right.,
\\
\b\varphi^{\b E}_5 =&
\left\{ \begin{tabular}{l l}
            $\b\omega_{34}^{T_1}$ & on $T_1$\vspace{0.2cm}\\
            $0$ & on $T_2$
            \end{tabular}\right.,
\quad &&  
\b\varphi^{\b E}_6 =
\left\{ \begin{tabular}{l l}
            $\lambda_1^{T_1}\grads\lambda_2^{T_1}$ & on $T_1$\vspace{0.2cm}\\
            $\b\omega_{23}^{T_2}+\lambda_2^{T_2}\grads\lambda_1^{T_2}$ & on $T_2$
            \end{tabular}\right.,
            \\
  \b\varphi^{\b E}_7 =&
  \left\{ \begin{tabular}{l l}
            $\b\omega_{14}^{T_1}$ & on $T_1$\vspace{0.2cm}\\
            $\b\omega_{24}^{T_2}$ & on $T_2$
            \end{tabular}\right.,
            \quad &&
\b\varphi^{\b E}_8 =
\left\{ \begin{tabular}{l l}
            $0$ & on $T_1$\vspace{0.2cm}\\
            $\b\omega_{34}^{T_2}$ & on $T_2$
            \end{tabular}\right.,
\end{alignat*}
i.e., $N_i^{\b E}(\b\varphi_j^{\b E}) = \delta_{ij}$.
\end{lemma}
\begin{proof}
The proof is similar to that for the $H^1$ case.
First, we can verify directly from the definition that each function $\b\varphi_i^{\b E}$ 
belongs to $\b E(K)$ and satisfies 
$N_i^{\b E}(\b\varphi_j^{\b E})=\delta_{ij}$. Hence, 
\begin{align*}
\dim \b E(K)\ge 
\dim \left(\mathrm{span}\{\b\varphi_i^{\b E}:\;1\le i\le 8\}\right) = 8.
\end{align*}

Next, we prove unisolvence of the set of degrees of freedom $\Sigma^E_K$.
Let $\b u\in \b E(K)$ be such that $N_i^{\b E} (\b u) = 0$ for $i= 1,2,\cdots,8$.
In other words, the tangential trace of $\b u$ vanishes on all edges of $K$. 
Since the tangential trace of $\b u$ on a face belongs to
the lowest-order \NED\!\! space of the first kind,
we immediately have 
the tangential traces of $\b u$ vanishes on all faces of $K$.
Hence, $\b u\restrict{T_\ell}\in \pol_1(T_\ell)^3$, for $\ell=1,2$, and $u$ has vanishing tangential traces on three faces of $T_\ell$.
This implies that $\b u=\b 0$.
Hence, $\dim \b E(K) = 8$, and the set $\Sigma_K^{\b E}$ is unisolvent. 
\end{proof}

\subsection{The $H(\mathrm{div})$-conforming finite element}
We present the lowest-order $H(\mathrm{div})$-conforming composite finite element on 
a regular pyramid.
To simplify notation, let faces be labelled such that 
\begin{align*}
\fcI_{1} = \fc_{125}, \;
\fcI_{2} = \fc_{145},\;
\fcI_{3} = \fc_{235},\;
\fcI_{4} = \fc_{345},\;
\fcI_{5} = \fc_{1234}.
\end{align*}

\begin{definition}
\label{Hdiv-def}
The lowest-order 
composite $H(\mathrm{div})$-conforming finite element is  defined by 
$(K, \b V(K), \Sigma^{\b V}_K)$ where
\begin{itemize}
\item $K$ is a regular pyramid with parallelogram base,
 \item 
 \begin{alignat}{3}
\label{Hdiv}
\bld V(K) :=&\; \{\b v \in H(\mathrm{div},K):&& \;
\b v\restrict{T_\ell}\in \pol_0(T_\ell)^3\oplus \bld x\pol_0(T_\ell),\;\; \ell = 1,2,\\
& &&\;\;\b v\cdot \b n\restrict{\fc}\in \pol_{0}(\fc),\quad\forall \fc\in \FC(K)\},\nonumber
\end{alignat}
where $\b n$ is the unit outward normal vector on face $\fc$.
\item 
$\Sigma_K^{\b V}=\mathrm{span}\{N_i^{\b V}:\;1\le i\le 6\}$, where
\[
N_i^{\b V}:= N_i^{\fcI}: \b v\longrightarrow \int_{\fcI_i} \b v\cdot \b n\,\mathrm{ds} \quad \forall\, \fcI_i\in \FC_K,
\]
are the total normal fluxes over each face $\fcI_i\in \FC_K$,
and 
\[
N_6^{\b V}:= N_0^{\mathrm{C}}: \b v\longrightarrow 
\int_{T_1} \divs\b v\,\mathrm{dx}-
 \int_{T_2} \divs\b v\,\mathrm{dx}.
\]
is a cell-based degrees of freedom.
\end{itemize}
\end{definition}

We denote the  {\it Whitney face form} on the tetrahedral $T_\ell$, $\ell=1,2$ as
\[
\b\chi_{ijk}^{T_\ell}: = \lambda_i^{T_\ell}\grads\lambda_j^{T_{\ell}}\times\grads\lambda_k^{T_{\ell}}
+\lambda_j^{T_\ell}\grads\lambda_k^{T_{\ell}}\times\grads\lambda_i^{T_{\ell}}
+\lambda_k^{T_\ell}\grads\lambda_i^{T_{\ell}}\times\grads\lambda_j^{T_{\ell}} .
\]
A basis for the space $\b V(K)$ is given as follows.
\begin{lemma}
\label{lemma-3}
The triplet  $(K,\b V(K), \Sigma^{\b V}_K)$ is a finite element of dimension 6.
The set $\{\b\psi^{\b V}_i:\;i=1,2,\cdots,6\}$ is the nodal basis dual to 
$\Sigma^{\b V}_K$, where 
 \begin{alignat*}{3}
 \b\psi^{\b V}_1 =&
\left\{ \begin{tabular}{l l}
            $2\,\b\chi_{234}^{T_1}-\b\chi_{124}^{T_1}$ & on $T_1$\vspace{0.2cm}\\
            $\b\chi_{124}^{T_2}$ & on $T_2$
            \end{tabular}\right.,
\quad && 
  \b\psi^{\b V}_2 =
  \left\{ \begin{tabular}{l l}
            $\b\chi_{124}^{T_1}$ & on $T_1$\vspace{0.2cm}\\
            $2\,\b\chi_{143}^{T_2}-\b\chi_{124}^{T_2}$ & on $T_2$
            \end{tabular}\right.,
\\
  \b\psi^{\b V}_3 =&\left\{ \begin{tabular}{l l}
            $2\,\b\chi_{143}^{T_1}-\b\chi_{124}^{T_1}$ & on $T_1$\vspace{0.2cm}\\
            $\b\chi_{124}^{T_2}$ & on $T_2$
            \end{tabular}\right.,
\quad &&  
\b\psi^{\b V}_4 =
  \left\{ \begin{tabular}{l l}
            $\b\chi_{124}^{T_1}$ & on $T_1$\vspace{0.2cm}\\
            $2\,\b\chi_{234}^{T_2}-\b\chi_{124}^{T_2}$ & on $T_2$
            \end{tabular}\right.,
\\
\b\psi^{\b V}_5 =&
\left\{ \begin{tabular}{l l}
            $\b\chi_{132}^{T_1}$ & on $T_1$\vspace{0.2cm}\\
            $\b\chi_{132}^{T_2}$ & on $T_2$
            \end{tabular}\right.,
\quad &&  
\b\psi^{\b V}_6 =
\left\{ \begin{tabular}{l l}
            $\b\chi_{124}^{T_1}$ & on $T_1$\vspace{0.2cm}\\
            $-\b\chi_{124}^{T_2}$ & on $T_2$
            \end{tabular}\right.,
\end{alignat*}
i.e., $N_i^{\b V}({\b\psi}^{\b V}_j) = \delta_{ij}$.
\end{lemma}
\begin{proof}
The proof is again similar to the previous two cases, and is therefore omitted.
%
%
%
\end{proof}

\subsection{The $L^2$-conforming finite element}
We present the lowest-order $L^2$-conforming composite finite element on 
a regular pyramid.

\begin{definition}
\label{L2-def}
The lowest-order 
composite $H(\mathrm{div})$-conforming finite element is defined by 
$(K, \b V(K), \Sigma^{\b V}_K)$ where
\begin{itemize}
\item $K$ is a regular pyramid with parallelogram base,
 \item 
 \begin{alignat}{3}
\label{L2}
W(K) :=&\; \{\psi \in L^2(K):&& \;
\psi\restrict{T_\ell}\in \pol_0(T_\ell),\;\; i = 1,2\}.
\end{alignat}
\item 
$\Sigma_K^{W}=\mathrm{span}\{N_i^{W}:\;i=1,2\}$, where
\begin{align*}
N_1^{W} : v\longrightarrow &\;
\int_{T_1} v\,\mathrm{dx}+
 \int_{T_2} v\,\mathrm{dx},\\ 
N_2^{W} : v\longrightarrow &\;
\int_{T_1} v\,\mathrm{dx}-
 \int_{T_2} v\,\mathrm{dx}.
\end{align*}
\end{itemize}
\end{definition}

The proof of the following result is left as an exercise for the reader.
\begin{lemma}
\label{lemma-4}
The triplet  $(K,W(K), \Sigma^{W}_K)$ is a finite element of dimension 2.
The set $\{\psi^{W}_i:\;i=1,2\}$ is the nodal basis dual to 
$\Sigma^{W}_K$, where 
 \begin{alignat*}{3}
 \psi^W_1 =& 
 \divs \b\psi^{\b V}_5 =
 \left\{ \begin{tabular}{l l}
            $\frac{1}{2|T_1|}$ & on $T_1$\vspace{0.2cm}\\
            $\frac{1}{2|T_2|}$ & on $T_2$
            \end{tabular}\right.,
\quad && 
  \psi^W_2 =\divs \b\psi^{\b V}_6=
  \left\{ \begin{tabular}{l l}
            $\frac{1}{2|T_1|}$ & on $T_1$\vspace{0.2cm}\\
            $-\frac{1}{2|T_2|}$ & on $T_2$
            \end{tabular}\right..
\end{alignat*}
\end{lemma}

\subsection{Exact sequence property}
The finite element family defined above is an exact sequence with a commuting diagram property on a regular pyramid.

Firstly, we define the following set of {\em nodal interpolation operators} defined
on the elements in the usual way:
\begin{alignat*}{6}
 \varPi_S(v) = & \sum_{i=1}^5 N_i^S(v)\phi_i^S, && \quad\forall\, v\in H^{r}(K),\\
\b\varPi_E(\b v) = &\; \sum_{i=1}^8 N_i^{\b E}(\b v)\b\varphi_i^{\b E},&& \quad\forall\,\b v\in H^{r-1}(\mathrm{curl},K),\\
\b \varPi_V(\b v) = & \sum_{i=1}^6 N_i^{\b V}(\b v)\b\psi_i^{\b V},&& \quad\forall\, \b v\in H^{r-1}(\mathrm{div},K),\\
 \varPi_W(\b v) = &\; \sum_{i=1}^2 N_i^W(v)\psi_i^W,&& \quad\forall\, v\in L^2(K),
 \end{alignat*}
 where $r>3/2$. By $H^{r-1}(\mathrm{curl}, K)$, we mean the space of all vector-valued functions in $H^{r-1}(K)^3$ whose curl is in
$H^{r-1}(K)^3$.

The results are summarized below.

\begin{theorem}
 \label{thm-commute}
 Let $K$ be a regular pyramid with a parallelogram base. Then,
 
{\em\sf{(a)}} the following sequence is exact:
\begin{align*}
 \bR
\overset{id}{\longrightarrow}
S(K)
\overset{\grads}{\longrightarrow} 
\b E( K)
\overset{\curls}{\longrightarrow} 
\b V(K)
\overset{\divs}{\longrightarrow} 
W(K)
\overset{}{\longrightarrow} 0.
\end{align*}

{\em \sf{(b)}}
The following diagram {\em commutes}:
\begin{alignat*}{2}
\begin{tabular}{l c c c c c l l}
$H^r({K})$&
$\overset{\grads}{\longrightarrow} $&
$H^{r-1}(\mathrm{curl},K)\!\!$&
$\overset{\curls}{\longrightarrow} $&
$H^{r-1}(\mathrm{div},K)\!\!$&
$\overset{\divs}{\longrightarrow} $&
$H^{r-1}(K)\;$\\
$\;\;\;\;\downarrow {\scriptstyle \varPi_S}$ & &
$\downarrow{\scriptstyle\b \varPi_E}$
& &$\downarrow{\scriptstyle\b \varPi_V}$
& &
$\;\;\downarrow{\scriptstyle \varPi_W}$
\\
$S(K)\!\!$&
$\overset{\grads}{\longrightarrow}$ &
$\b E(K)\!\!$&
$\overset{\curls}{\longrightarrow}$ &
$\b V(K)\!\!$&
$\overset{\divs}{\longrightarrow}$ &
$W(K)$.
\end{tabular}
\end{alignat*}
That is,
\begin{alignat*}{3}
 \b\varPi_E\grads v = &\; \grads \varPi_S v, \quad && \forall \, v\in H^r(K),\\
 \b\varPi_V\curls\b v = &\; \curls \b\varPi_E\b v, \quad && \forall \,\b v\in H^{r-1}(\mathrm{curl},K),\\
\varPi_W\divs v = &\; \divs\b \varPi_V\b v, \quad && \forall\,\b v \in H^{r-1}(\mathrm{div},K).
\end{alignat*}

{\em  \sf{(c)}}
The nodal interpolation operators have the following approximation properties:
\begin{itemize}
 \item [(1)]
If $v\in H^s(K)$ with $3/2<s\le 2$, then
\begin{alignat*}{3}
 \|v-\varPi_S v\|_0\preceq &\; h_K^{s}\|v\|_s,
 \;\;
  \|\grads(v-\varPi_S v)\|_0\preceq &\; h_K^{s-1}\|v\|_s.
\end{alignat*}
 \item [(2)]
If $\b v\in H^s(\mathrm{curl},K)$ with $1/2<s\le 1$, then
\begin{alignat*}{3}
 \|\b v-\b\varPi_E \b v\|_0\preceq &\; h_K^{s}
 (\|\b v\|_s+\|\curls \b v\|_s), \;\; 
 \|\curls(\b v-\b\varPi_E \b v)\|_0\preceq &\; h_K^{s}
 \|\curls \b v\|_s.
\end{alignat*}
 \item [(3)]
If $\b v\in H^s(\mathrm{div},K)$ with $1/2<s\le 1$, then
\begin{alignat*}{3}
 \|\b v-\b\varPi_V \b v\|_0
\preceq &\; h_K^{s}
\|\b v\|_s, \;\;
 \|\divs(\b v-\b\varPi_V \b v)\|_0
\preceq &\; h_K^{s}
\|\divs\b v\|_s.
\end{alignat*}
 \item [(4)]
If $v\in H^s(K)$ with $1/2<s\le 1$, then
\begin{alignat*}{3}
 \|v-\varPi_W v\|_0
\preceq &\; h_K^{s}\|v\|_s.
\end{alignat*}
\end{itemize}
Here $h_K$ is the diameter of the pyramid $K$. 
\end{theorem}
\begin{proof}
 The exactness of the above sequence in part {\sf (a)} is  a direct consequence of the exactness of the following sequence on a tetrahedron $T$:
\begin{align*}
 \bR
\overset{id}{\longrightarrow}
\pol_2(T)
\overset{\grads}{\longrightarrow} 
\pol_1(T)^3
\overset{\curls}{\longrightarrow} 
\pol_0(T)^3\oplus \b x\,\pol_0(T)
\overset{\divs}{\longrightarrow} 
\pol_0(T)
\overset{}{\longrightarrow} 0,
\end{align*}
and the trace properties of the corresponding composite spaces.

The proof of commutativity in part {\sf(b)} follows from integration by parts.

The composite spaces contain (sufficiently rich) polynomial functions on the whole pyramid; specifically
\begin{alignat*}{3}
\pol_1(K)\subset&\; S(K),\quad  && 
\pol_0(K)^3\oplus \b x\times \pol_0(K)^3\subset\; \b E(K),\\
\pol_0(K)^3\oplus \b x\,\pol_0(K)\subset&\; \b V(K),
\quad&& \pol_0(K)\subset\; W(K).
\end{alignat*}
The approximation estimates in part {\sf (c)} then follow from a standard scaling argument and the Bramble-Hilbert Lemma;
see, e.g.,  \cite[Chapter 5]{Monk03}. 
\end{proof}
\begin{remark}
Similar approximation properties as Theorem 6(c) hold for the tetrahedral and hexahedral cases when the element 
$K$ is a tetrahedron or a parallelepiped with the corresponding finite element spaces given in Table 1--4; 
see \cite[Chapter 5,6]{Monk03}.
The degrees of freedom on a tetrahedron or a parallelepiped
are the usual nodal evaluations for $H^1$-conforming finite elements, 
the tangential edge integrals for $H(\mathrm{curl})$-conforming finite elements, the normal face integrals for $H(\mathrm{div})$-conforming finite elements, and 
the cell integral for $L^2$-conforming finite elements,
and the corresponding nodal interpolators are defined in the usual way.
\end{remark}
\begin{remark}
An extensive analysis on commuting diagrams can be found in Bossavit \cite{Bossavit98} for nodal interpolation operators,
Cl\'em\'ent-type { quasi-interpolation} operators in Sch\"oberl \cite{Schoberl08}, 
and projection-based interpolatants in {Demkowicz $\&$ Buffa} \cite{DemkowiczBuffa05}. 
\end{remark}

\section{Lowest-order finite elements on Tetrahedral-Hexahedral-Pyramidal ({\sf THP}) partitions}
\subsection{Composite finite elements on non-affine pyramids}
The finite elements and basis functions constructed above can be used on conforming pyramidal meshes 
consisting of {\it regular pyramids} with a parallelogram base. 
However, in practice the hexahedral elements in the mesh will have faces that are quadrilaterals. Therefore, we now turn our attention to 
%
the case of {\it physical pyramids} with a bilinear 
B\'ezier patch base (which can be attached to a trilinear-mapped hexahedron). 
We recall the composite quadratic mapping defined in \cite{AinsworthDavydovSchumaker16}.
A pyramid is uniquely determined by its five vertices $\{\vt_i:1\le i\le5\}$
labeled as in Figure \ref{fig-pyramid}. The quadrilateral base  is
a non-degenerate bilinear B\'ezier patch in $\bR^3$ given by 
\[
\{
\lambda_1 \lambda_2\vt_1
+\lambda_1 (1-\lambda_2)\vt_2
+(1-\lambda_1)(1- \lambda_2)\vt_3
+(1-\lambda_1) \lambda_2\vt_4:(\lambda_1,\lambda_2)\in [0,1]^2
 \}.\]
The pyramid has four straight-sided triangular faces and a (possibly) non-planer quadrilateral face.
Moreover, the pyramid is regular if and only if $\vt_P = \b 0$ where
 $\vt_P := \vt_{1}-\vt_{2}+\vt_3-\vt_4$.

Let $\widehat{K}$ denote the reference pyramid with vertices at
\begin{align}
\label{pyr}
 \hat{\vt}_1 = (0,0,0), 
\hat{\vt}_2 = (1,0,0), 
\hat{\vt}_3 = (1,1,0), 
\hat{\vt}_4 = (0,1,0), 
\hat{\vt}_5 = (1/2,1/2,1),
\end{align}
and let $\widehat{T}_1$ and $\widehat{T}_2$ 
be the two tetrahedra obtained by splitting $\widehat{K}$
as shown in Figure \ref{fig-pyramid}.

The mapping from $\widehat{K}$ to the physical pyramid $K$ is given as follows:
\begin{align}
\label{mapping-pyr}
 \Phi_K(\widehat{\b x}) = 
  \left\{ \begin{tabular}{l l}
            $\lambda_1^{\widehat{T}_1} \vt_3+\lambda_2^{\widehat{T}_1} \vt_1+
            \lambda_3^{\widehat{T}_1} \vt_2+\lambda_4^{\widehat{T}_1} \vt_5-
            \lambda_1^{\widehat{T}_1} \lambda_2^{\widehat{T}_1} \vt_P$ & if $\widehat{\b x}\in \widehat{T}_1$\vspace{0.2cm}\\
            $\lambda_1^{\widehat{T}_2} \vt_1+\lambda_2^{\widehat{T}_2} \vt_3+
            \lambda_3^{\widehat{T}_2} \vt_4+\lambda_4^{\widehat{T}_2} \vt_5-
            \lambda_1^{\widehat{T}_2} \lambda_2^{\widehat{T}_2} \vt_P$ & if $\widehat{\b x}\in\widehat{T}_2$
            \end{tabular}\right.,
\end{align}
where $\b\lambda^{\widehat{T}_\ell}=\{\lambda_1^{\widehat{T}_\ell},\cdots, \lambda_4^{\widehat{T}_\ell}\}$ are 
the barycentric coordinates of $\widehat{\b x} = (\hat x_1, \hat x_2,\hat x_3)$ on the reference tetrahedral $\widehat{T}_\ell$, i.e.,
\begin{alignat*}{5}
&\lambda_1^{\widehat{T}_1} = \; \hat x_2-\frac{\hat x_3}{2}, 
&&\lambda_2^{\widehat{T}_1} =\; 1-\hat x_1-\frac{\hat x_3}{2}, \;\;\;
&\lambda_3^{\widehat{T}_1} =\; \hat x_1-\hat x_2, \;\;
&&\lambda_4^{\widehat{T}_1} =\;\hat x_3, \\
&\lambda_1^{\widehat{T}_2} =\; 1-\hat x_2-\frac{\hat x_3}{2}, \;\;\;
&&\lambda_2^{\widehat{T}_2} =\; \hat x_1-\frac{\hat x_3}{2}, \;\;
&\lambda_3^{\widehat{T}_2} =\; \hat x_2-\hat x_1, \;\;
&&\lambda_4^{\widehat{T}_2} =\;\hat x_3.
\end{alignat*}


The finite elements on the physical pyramid $K$ are defined 
in terms of those on the reference pyramid $\widehat K$ in the usual way \cite{Monk03}; 
see \eqref{space-mapping-1}--\eqref{space-mapping-4} in Section 1.
\subsection{The {\sf THP} partition and mesh refinement}
A {\sf THP} partition
is a conforming mesh consists of affine tetrahedra, pyramids with bilinear B\'ezier patch bases, and trilinear-mapped hexahedra
\cite[Definition 2.3]{AinsworthDavydovSchumaker16}. 

More generally, 
we consider a series of {\sf THP} partitions of a polyhedral domain obtained via uniform refinements of a given initial 
 {\sf THP} partition where the individual elements are refined using the rules indicated in Figure \ref{fig:tet}--\ref{fig:pyr}.

\begin{figure}[!ht]
\centerline{
\psfig{file=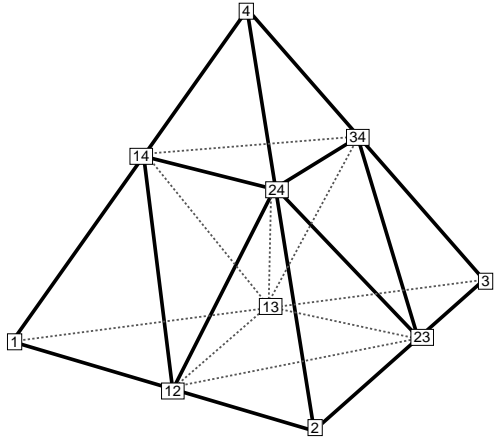,width=3.2in}
}
\caption{
The original tetrahedral with vertices $(1,2,3,4)$ is divided into
8 tetrahedra with vertices
$(1,12,13,14)$,  
$(12,2,23,24)$, $(13,23,3,34)$,
$(14,24,34,4)$, $(14,12,13,24)$, $(13,12,23,24)$, $(13,23,34,24)$, 
$(13,34,14,24)$, where 
the node at vertex $ij$ is placed at the centroid of 
nodes at  vertices $i$ and $j$.
}
\label{fig:tet}
\end{figure}


\begin{figure}[!ht]
\centerline{
\psfig{file=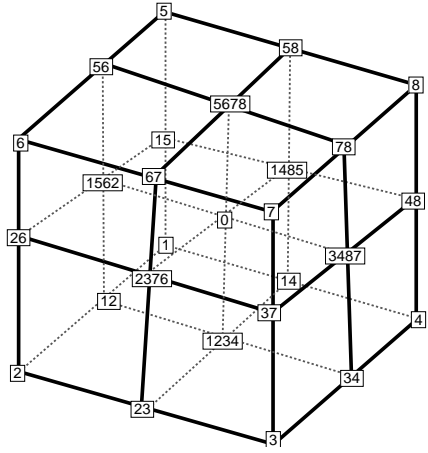,width=3.2in}
}
\caption{
 The original hexahedron with vertices $(1,2,3,4,5,6,7,8)$ is divided into
8 hexahedra with vertices
$(15,1562,0,1485,5,56,5678,58)$, $(1485,0,3487,48,58,5678,78,8)$,
$(1562,26,2376,0,56,6,67,5678)$, $(0,2376,37,3487,5678,67,7,78)$,
$(1,12,1234,14,15,1562,0,1485)$, $(14,1234,34,4,1485,0,3487,48)$,
$(12,2,23,1234,1562,26,2376,0)$, $(1234,23,3,34, 0,2376,37,3487)$,
where 
the new node at vertex $ijk\ell$ is placed at the centroid of nodes with 
 vertices $i$, $j$, $k$, and $\ell$,
and the new node at vertex $0$ is placed at the centroid of the nodes for all eight vertices of the original hexahedron.
}
\label{fig:hex}
\end{figure}

\begin{figure}[!ht]
\centerline{
\psfig{file=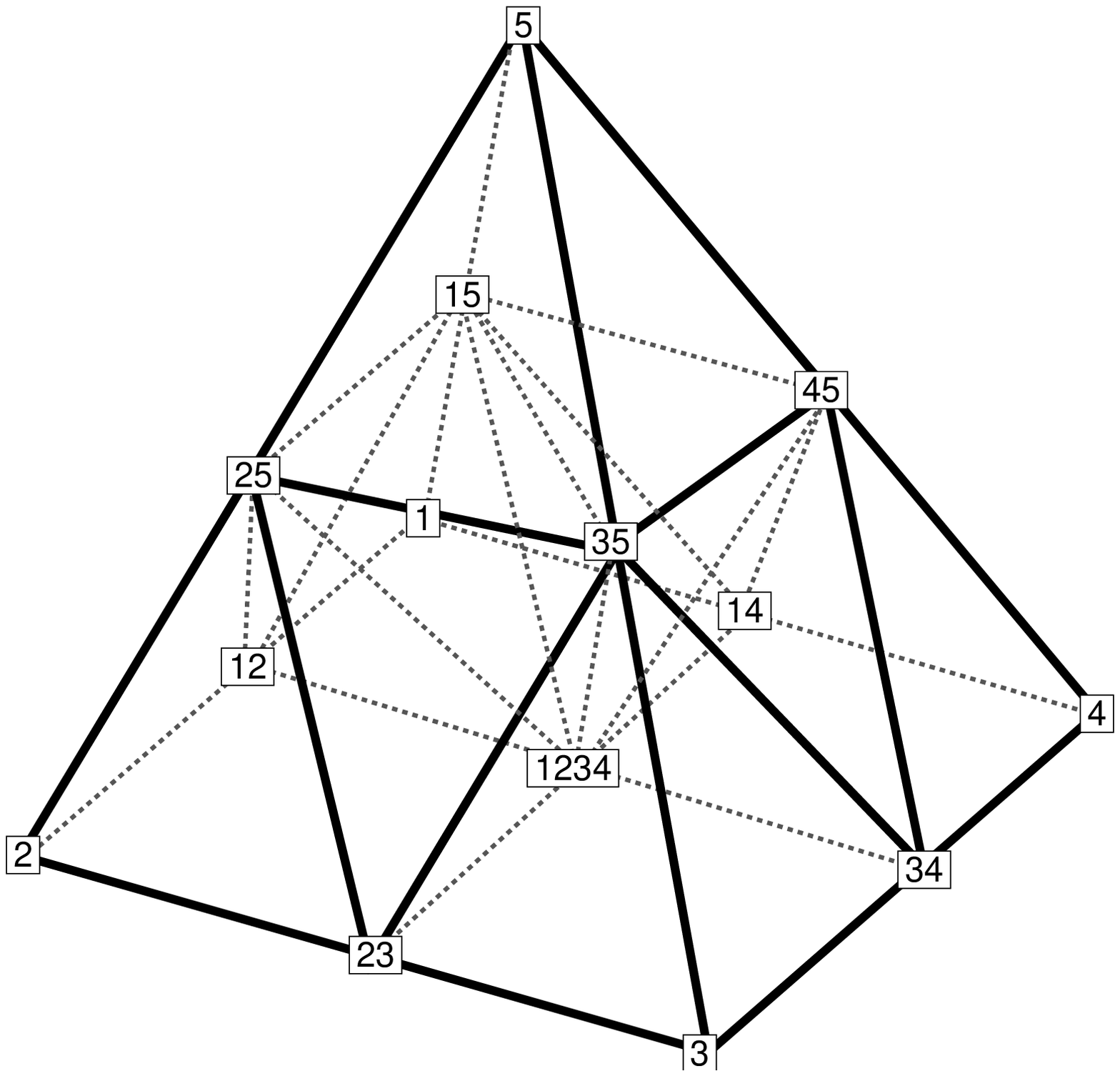,width=3.2in}
}
\caption{
The original pyramid with vertices $(1,2,3,4,5)$ is divided into
4 pyramids and 8 tetrahedra with vertices
$(1,12,1234,14,15)$, $(12,2,23,1234,25)$, 
$(1234,23,3,34,35)$, $(14,1234,34,4,45)$,
$ (12,15,25,1234)$, $(23,35,25,1234)$, $(34,45,35,1234)$, $(14,45,15,1234)$,
$(15,25,35, 5)$, $(15,35,45,5)$, $(15,35,25,1234)$, $(15,45,35,1234)$.
}
\label{fig:pyr}
\end{figure}

It is clear that a uniform refinement of a {\sf THP} partition is again a {\sf THP} partition.
\begin{remark}
There are three choices of uniform tetrahedral refinements. 
Instead of choosing the diagonal edge by connecting the nodes $13$ and $24$ as shown in 
 Figure \ref{fig:tet}, we could connect the nodes 
 $14$ and ${23}$, or the nodes $12$ and $34$.
 We shall follow \cite[Page 1138]{Ong94} to choose a particular diagonal, which guarantees 
the non-degeneracy of the tetrahedra in the series of {\sf THP} partitions.
\end{remark}

\subsection{The global finite elements}
With the finite elements on pyramids in Section 3, we are ready to construct a global 
lowest-order finite element exact sequence on a {\sf THP} partition $\mathcal{T}_h$ combining with the 
finite elements on tetrahedra and hexahedra listed in 
\eqref{space-mapping-1}--\eqref{space-mapping-4} (and Table \ref{table-H1}--\ref{table-L2}) in Section 1:
\begin{alignat*}{4}
  S_{h,1} =&
 \{ v\in H^1(\Omega):
&&\;\;\; v\restrict{K} \in S(K),\;\;\;\forall\, K\in\mathcal{T}_h
\}\\
  \b E_{h,0} =&
 \{\b v\in H(\mathrm{curl},\Omega):
&& \;\;\;\b v\restrict{K} \in \b E(K),\;\;\;\forall\, K\in\mathcal{T}_h
\}\\
  \b V_{h,0} =&
 \{\b v\in H(\mathrm{div},\Omega):
&&\;\;\; \b v\restrict{K} \in \b V(K),\;\;\;\forall\, K\in\mathcal{T}_h
\}\\
  W_{h,0} =&
 \{v\in L^2(\Omega):
&&\;\;\; v\restrict{K} \in W(K),\;\;\;\forall\, K\in\mathcal{T}_h
\}.
\end{alignat*}

We define the global interpolation operators $\varPi_S, \b\varPi_E,\b\varPi_V,$ and $\varPi_W$, restricted to each element
to be the image of the corresponding interpolant defined on the reference element under the mapping used to define the 
finite elements on the physical element in terms of the finite elements on the reference element.

The global finite elements give an exact sequence on the {\sf THP} partition $\mathcal{T}_h$ and satisfy the 
commuting diagram property.

\begin{theorem}
\label{thm-exact-g}
{\em\sf{(a)}}
The sequence 
 \[
 \bR
\overset{id}{\longrightarrow}
S_{h,1}
\overset{\grads}{\longrightarrow} 
\b E_{h,0}
\overset{\curls}{\longrightarrow} 
\b V_{h,0}
\overset{\divs}{\longrightarrow} 
W_{h,0}
\overset{}{\longrightarrow} 0,
 \]
is exact.\newline
{\em\sf{(b)}}
The following diagram {\em commutes}:
\begin{alignat*}{2}
\begin{tabular}{l c c c c c l l}
$H^r({\Omega})$&
$\overset{\grads}{\longrightarrow} $&
$H^{r-1}(\mathrm{curl},\Omega)\!\!$&
$\overset{\curls}{\longrightarrow} $&
$H^{r-1}(\mathrm{div},\Omega)\!\!$&
$\overset{\divs}{\longrightarrow} $&
$H^{r-1}(\Omega)\;$\\
$\;\;\;\;\downarrow {\scriptstyle \varPi_S}$ & &
$\downarrow{\scriptstyle\b \varPi_E}$
& &$\downarrow{\scriptstyle\b \varPi_V}$
& &
$\;\;\downarrow{\scriptstyle \varPi_W}$
\\
$S_{h,1}\!\!$&
$\overset{\grads}{\longrightarrow}$ &
$\b E_{h,0}\!\!$&
$\overset{\curls}{\longrightarrow}$ &
$\b V_{h,0}\!\!$&
$\overset{\divs}{\longrightarrow}$ &
$W_{h,0}$,
\end{tabular}
\end{alignat*}
where $r>3/2$.
\end{theorem}

\subsection{Approximation properties}
Optimal approximation properties of the above
global finite elements on {\it affine} {\sf THP} partitions 
follow directly by summing the local estimates provided in
part (c) of Theorem \ref{thm-commute} (and Remark 1).
Here we prove optimal approximation properties, following \cite{RannacherTurek92, IngramWheelerYotov10}, 
of the global finite elements on 
a series of {\sf THP} partitions obtained via uniform refinements of an initial coarse {\sf THP} partition. 

The finite elements on a physical element $K$ (tetrahedron, hexahedron, or pyramid)
are defined via proper mappings from the reference element $\widehat K$ (reference tetrahedron \eqref{tet}, reference hexahedron \eqref{cube}, or reference pyramid \eqref{pyr}). 
The {\it affine} tetrahedral mapping is
\begin{subequations}
\begin{align}
\label{mapping-tet}
 \Phi_K (\widehat x) = &\; \sum_{i=1}^4 \vt_i\phi_i,
\end{align}
where $\{\vt_i\}_{i=1}^4$ are the vertices of the physical tetrahedron, and
$\{\phi_i\}_{i=1}^4$ are the four linear nodal bases on the reference tetrahedron given in Table \ref{table-H1};
the {\it trilinear} hexahedral mapping is
\begin{align}
\label{mapping-hex}
 \Phi_K (\widehat x) = &\; \sum_{i=1}^8 \vt_i\phi_i,
\end{align}
\end{subequations}
where $\{\vt_i\}_{i=1}^8$ are the vertices of the physical hexahedron, and
$\{\phi_i\}_{i=1}^8$ are the eight trilinear nodal bases  on the reference hexahedron given in Table \ref{table-H1};
and the (composite) pyramidal mapping is given in \eqref{mapping-pyr}.

We have the following bounds on the derivatives of the mappings on a series of uniformly refined {\sf THP} partitions.
\begin{lemma}
  \label{lemma-mapping-1}
 Let $\{\mathcal{T}_h\}_{h\downarrow 0}$ be a series of uniformly refined {\sf THP} partitions of a polyhedral domain.
 Then, 
\[
\max_{i,j\in\{1,2,3\}}\left  |\frac{\partial^2\, \Phi_K}{\partial \hat x_i\partial \hat x_j}\right|_{\infty, \widehat K}
\preceq h_K^2
\]
for all element $K\in \mathcal{T}_h$.
\end{lemma}
\begin{proof}
We consider each type of elements in turn.

When the element $K$ is a tetrahedron, the mapping $\Phi_K$ is affine, 
and the inequality holds trivially.

When the element $K$ is a pyramid obtained by mapping a reference pyramid using 
the piecewise quadratic mapping \eqref{mapping-pyr}, 
the second derivative of the mapping $\Phi_K$ is a constant in each of 
the tetrahedra of which it is composed, and moreover, on each such tetrahedra, we have
\begin{align*}
 \left|\frac{\partial^2\, \Phi_K}{\partial \hat x_i\partial \hat x_j}\right|\le |\vt_P|,\quad\quad
 \text{ for all } i,j\in\{1,2,3\}.
\end{align*}
Here $\vt_P =\vt_{1}-\vt_{2}+\vt_3-\vt_4$, and $\{\vt_i\}_{i=1}^4$ are the four vertices for the quadrilateral base
of the pyramid as shown in Figure \ref{fig-pyramid}.
Hence, in order to prove the bound in Lemma \ref{lemma-mapping-1}, it suffices to show that $|\vt_P|\preceq h_K^2$.

We consider two consecutive {\sf THP} partitions, $\mathcal{T}_h$ and $\mathcal{T}_{h/2}$, where 
$\mathcal{T}_{h/2}$ is obtained from $\mathcal{T}_h$ via a uniform refinement.

Let $K_{m}\in \mathcal{T}_h$ be a pyramid with vertices $(1,2,3,4,5)$, and 
$K_{c}\in \mathcal{T}_{h/2}$ be one of its four child pramids.
Without loss of generality, we take the vertices of $K_{c}$ to be 
$
 (1, 12,1234,14,5)$
as shown in Figure \ref{fig:pyr}.
We have 
\begin{align*}
  \vt_P(K_c) = &\;
  \vt_{1} - \vt_{12} + 
 \vt_{1234}-\vt_{14}\\
 = &\; \vt_{1}-\frac{1}{2}(\vt_{1}+\vt_{2})  +\frac{1}{4}(\vt_{1}+\vt_{2}+\vt_3+\vt_4)
 - \frac12(\vt_{1}+\vt_{4})\\
 = &\;\frac{1}{4}(\vt_{1}-\vt_{2}+\vt_3-\vt_4)\;= \;\frac{1}{4}\vt_P(K_m).
\end{align*}
This means when a pyramid $K_m$ is subdivided, the size of each child pyramid is $1/2$ that of the mother whereas
the quantity $\vt_P(K_c)$ is $(1/2)^2$ of the corresponding quantity $\vt_P(K_m)$. It follows by induction that 
$|\vt_P|\preceq h_K^2$ as required.

When the element $K$ is a hexahedron obtained from the trilinear mapping \eqref{mapping-hex}, 
each of the second cross derivatives $\frac{\partial^2\, \Phi_K}{\partial \hat x_i\partial \hat x_j}$, $i\not=j$,  of the mapping $\Phi_K$ is a linear function, 
and the second derivatives $\frac{\partial^2\, \Phi_K}{\partial \hat x_i\partial \hat x_i}$ vanish identically.
By symmetry, we only need to prove the bound for 
$\frac{\partial^2\, \Phi_K}{\partial \hat x_1\partial \hat x_2}$. 
Now, 
\begin{align*}
 \frac{\partial^2\, \Phi_K}{\partial \hat x_1\partial \hat x_2}
 = (\vt_1-\vt_2+\vt_3-\vt_4)(1-\widehat x_3)
 +(\vt_5-\vt_6+\vt_7-\vt_8)\widehat x_3,
\end{align*}
and so, 
denoting $\vt_P^{bot}(K) := \vt_1-\vt_2+\vt_3-\vt_4$ 
and $\vt_P^{top}(K) := \vt_5-\vt_6+\vt_7-\vt_8$, we have 
\[
\left| \frac{\partial^2\, \Phi_K}{\partial \hat x_1\partial \hat x_2}
\right|\le \max\left\{|\vt_P^{bot}|, |\vt_P^{top}|\right\}.
\]

As before, we consider two consecutive {\sf THP} partitions, $\mathcal{T}_h$ and $\mathcal{T}_{h/2}$, and 
show that the right hand side of the above inequality decreases (at least) by 
$(1/2)^2$ as the mesh size decrease by $1/2$. The result then follows by induction.

Let $K_{m}\in \mathcal{T}_h$ be a hexahedron with vertices $(1,2,3,4,5,6,7,8)$, and 
$K_{c}\in \mathcal{T}_{h/2}$ be one of its eight child hexahedra.
Without loss of generality, we take the vertices of $K_{c}$ to be 
$
 (15, 1562,0,1485,5, 56,5678,58)$
as shown in Figure \ref{fig:hex}.
We have 
\begin{align*}
  \vt_P^{bot}(K_c) = &\;
  \vt_{15} - \vt_{1562} + 
 \vt_{0}-\vt_{1485}\\
 = &\;\frac12(\vt_{1}+\vt_5)-
 \frac{1}{4}(\vt_{1}+\vt_{5}+\vt_6+\vt_2)  +
 \frac{1}{8}\sum_{i=1}^8\vt_{i}
 - \frac14(\vt_{1}+\vt_{4}+\vt_5+\vt_8)\\
 = &\;\frac{1}{8}(\vt_{1}-\vt_{2}+\vt_3-\vt_4
 +\vt_{5}-\vt_{6}+\vt_7-\vt_8)\\
 = &\;\frac{1}{8}\left(\vt_P^{bot}(K_m)+\vt_P^{top}(K_m)\right),\\
  \vt_P^{top}(K_c) = &\;
  \vt_{5} - \vt_{56} + 
 \vt_{5678}-\vt_{58}\\
 = &\;\vt_5-
 \frac{1}{2}(\vt_{5}+\vt_6)  +
 \frac14(\vt_{5}+\vt_{6}+\vt_7+\vt_8) - \frac12(\vt_{5}+\vt_8)\\
 = &\;\frac{1}{4}\vt_P^{top}(K_m).
\end{align*}
This implies that
\[
 \max\left\{|\vt_P^{bot}(K_c)|, |\vt_P^{top}(K_c)|\right\}\le 
  \frac14\max\left\{|\vt_P^{bot}(K_m)|, |\vt_P^{top}(K_m)|\right\}.
\]
Hence, we have $|\vt_P|\preceq h_K^2$ as desired.
\end{proof}

The following estimates for the Jacobian matrix and determinant of the mapping from the reference element to elements
in a family of {\sf THP} partitions will be useful:
\begin{lemma}
 \label{lemma-mapping-2}
 Let $\{\mathcal{T}_h\}_{h\downarrow 0}$ be a series of uniformly refined {\sf THP} partitions.
 Then, 
 \begin{subequations}
 \label{eq}
 \begin{align}
 \label{eq-1}
h_K\preceq|F_K|_{0,\infty, \widehat K} \preceq h_K,\quad \quad
h_K^3 \preceq|J_K|_{0,\infty, \widehat K} \preceq h_K^3,
 \end{align}
and 
\begin{align}
 \label{eq-2}
|F_K|_{1,\infty, \widehat K} \preceq h_K^2,\quad
|J_KF_K^{-1}|_{1,\infty, \widehat K} \preceq h_K^{3},
\quad
|J_K|_{1,\infty, \widehat K} \preceq h_K^{4},
\end{align}
 \end{subequations}
holds for all elements $K\in \mathcal{T}_h$.
\end{lemma}
\begin{proof}
Lemma \ref{lemma-mapping-1} means
that each hexahedral element in the series of {\sf THP} partitions is automatically
an {\it $h^2$-parallelepiped} in the sense defined in 
\cite[Chapter 3]{IngramWheelerYotov10}. 
Hence, the proof of Lemma \ref{lemma-mapping-2} 
for the hexahedral case can be found in \cite[Lemma 3.1]{IngramWheelerYotov10}. 
The proof for the pyramidal case is similar and omitted, and that for the tetrahedral case is trivial.
\end{proof}

The main approximation results of the finite elements are summarized below.
\begin{theorem}
 \label{thm:approx}
  Let $\{\mathcal{T}_h\}_{h\downarrow 0}$ be a series of uniformly refined {\sf THP} partitions of a polyhedral domain $\Omega$.
  Then, the nodal interpolation operators have the following approximation properties:
\begin{itemize}
 \item [(1)]
If $v\in H^2(\Omega)$,
then
\begin{alignat*}{3}
 \|v-\varPi_S v\|_0\preceq &\; h^2\|v\|_2,\;\;
  \|\grads(v-\varPi_S v)\|_0\preceq &\; h\|v\|_2,
\end{alignat*}
 \item [(2)]
If $\b v\in H^1(\mathrm{curl},\Omega)$, then 
\begin{alignat*}{3}
 \|\b v-\b\varPi_E \b v\|_0\preceq &\; h
 (\|\b v\|_1+\|\curls \b v\|_1), \;\; 
 \|\curls(\b v-\b\varPi_E \b v)\|_0\preceq &\; h
 \|\curls \b v\|_1.
\end{alignat*}
 \item [(3)]
If $\b v\in H^1(\mathrm{div},\Omega)$, then
\begin{alignat*}{3}
 \|\b v-\b\varPi_V \b v\|_0
\preceq &\; h
\|\b v\|_1, \;\;
 \|\divs(\b v-\b\varPi_V \b v)\|_0
\preceq &\; h
\|\divs\b v\|_1.
\end{alignat*}
 \item [(4)]
If $v\in H^1(\Omega)$, then
\begin{alignat*}{3}
 \|v-\varPi_W v\|_0
\preceq &\; h\|v\|_1.
\end{alignat*}
\end{itemize}
Here $h$ is the mesh size of the {\sf THP} partition $\mathcal{T}_h$, and 
all norms are evaluated on the domain $\Omega$.
\end{theorem}
\begin{proof}
We shall prove the estimates on a single element $K$, and sum to obtain 
the corresponding estimates on the composite mesh.

Let $\Phi_K$ be the mapping from a reference element $\widehat K$ to a physical element $K$, and 
$F_K$ and $J_K$ be the corresponding Jacobian matrix and determinant, respectively.

 Denote $\widehat {\varPi}_S$,
  $\widehat {\b \varPi}_E$,  $\widehat {\b \varPi}_V$,  and 
  $\widehat {\varPi}_W$, respectively,  be the nodal interpolators for the 
  $H^1$, $H(\mathrm{curl})$, $H(\mathrm{div})$, and $L^2$ finite elements on the reference element, respectively.

We first prove the estimates in the $L^2$-norm for each of the spaces $S_{h,1}$, $\b E_{h,0}$, $\b V_{h,0}$, and 
$W_{h,0}$.

The proof for $S_{h,1}$ follows line-by-line from \cite[Lemma 1]{RannacherTurek92}.
 For a given function $v\in H^2(K)$ on $K$, we associate the function 
 $\widehat v\equiv v\circ \Phi_K \in H^2(\widehat K)$.
Thanks to Theorem 6(c) part (1) and Remark 1, there holds
\[
 \|\widehat v-\widehat \varPi_S\widehat v\|_{\widehat K}\preceq 
  |\widehat v|_{2, \widehat K},
\]
and hence, observing the relation $\varPi_S v(\b x) = \widehat {\varPi}_S \widehat v(\widehat{\b x})$ and using the bound 
for the determinant $J_K$ in \eqref{eq-1}, it follows that 
\[
\|v- \varPi_S v\|_{K}\preceq h_K^{3/2} 
\|v- \varPi_S v\|_{\widehat K}\preceq h_K^{3/2} |\widehat v|_{2, \widehat K}. 
\]
To estimate the term on the right hand side, we note that 
\[
 |\widehat{\nabla}^2\widehat v| = 
  |\widehat{\nabla}(F_K^T \nabla v)|
  \preceq 
  |F_K|_{1,\infty} |\nabla v|+|F_K|^2_{0,\infty} |\nabla^2 v|
  \preceq h_K^2(|\nabla v|+|\nabla^2 v|),
\]
where the last inequality follows from the estimates in Lemma \ref{lemma-mapping-2}.
This yields
\[
 \|v- \varPi_S v\|_{K}\preceq h_K^2\|v\|_{2,K}.
\]

The proof for $\b E_{h,0}$ follows using a similar argument to the tetrahedral case
given in \cite[Theorem 5.41]{Monk03}.
For a given $\b v\in H^1(\mathrm{curl};K)$ on $K$, we associate the function 
 $\widehat{\b v}\equiv F_K^{T}\b v\circ \Phi_K \in H^1(\mathrm{curl};\widehat K)$.
By Theorem 6(c) part (2) and Remark 1, we have
\[
 \|\widehat {\b v}-\widehat {\b \varPi}_E\widehat{\b  v}\|_{\widehat K}\preceq 
  |\widehat{\b v}|_{1, \widehat K}+|\widehat{\nabla}\times\widehat{\b v}|_{1, \widehat K}.
\]
 Hence, using the relation $\b\varPi_E \b v(\b x) = F_K^{-T}\widehat {\b \varPi}_E \widehat{\b v}(\widehat{\b x})$ and using the bound 
for the Jacobian $F_K$ and determinant $J_K$ in \eqref{eq-1}, it follows that 
\[
\|\b v- \b\varPi_E \b v\|_{K}\preceq h_K^{1/2} 
\|\widehat{\b v}- \widehat{\b \varPi}_E\widehat{\b v}\|_{\widehat K}\preceq h_K^{1/2} 
(|\widehat{\b v}|_{1, \widehat K}+|\widehat{\nabla}\times\widehat{\b v}|_{1, \widehat K}). 
\]
The term on the right hand side can be estimated by noting that  
\[
 |\widehat{\nabla}\widehat{\b v}|   \preceq 
  |F_K|_{1,\infty} |\b v|+|F_K|^2_{0,\infty} |\nabla\b v|
  \preceq h_K^2(|\b v|+|\nabla\b v|),
\]
and then
\begin{align*}
  |\widehat{\nabla}(\widehat{\nabla}\times\widehat{\b v})|   \preceq&\; 
  |J_KF_K^{-1}|_{1,\infty} |\nabla\times\b v|+|J_KF_K^{-1}|_{0,\infty}|F_K|_{0,\infty} |\grads(\nabla\times\b v)|\\
  \preceq &\;h_K^3(|\nabla\times \b v|+|\nabla(\nabla\times\b v)|),
\end{align*}
thanks to the relation $\widehat{\nabla}\times\widehat{\b v} = J_KF_K^{-1}\nabla\times \b v$.
Hence,
\[
\|\b v- \b\varPi_E \b v\|_{K}\preceq
h_K(\|\b v\|_{1,K}+h_K \|\nabla \times \b v\|_{1,K}).
\]

Let us now prove the $L^2$-estimate for $\b V_{h,0}$.
For a given function $\b v\in H^1(\mathrm{div};K)$ on $K$, we associate the function 
 $\widehat{\b v}\equiv J_KF_K^{-1}\b v\circ \Phi_K \in H^1(\mathrm{curl};\widehat K)$. 
By Theorem 6(c) part (3) and Remark 1, we have
\[
 \|\widehat {\b v}-\widehat {\b \varPi}_V\widehat{\b  v}\|_{\widehat K}\preceq 
  |\widehat{\b v}|_{1, \widehat K}.
\]
 Hence, using the relation $\b\varPi_V \b v(\b x) = J_K^{-1}F_K\widehat {\b \varPi}_V \widehat{\b v}(\widehat{\b x})$ and using the bound 
for the Jacobian $F_K$ and determinant $J_K$ in \eqref{eq-1}, it follows that 
\[
\|\b v- \b\varPi_V \b v\|_{K}\preceq h_K^{-1/2} 
\|\widehat{\b v}- \widehat{\b \varPi}_V\widehat{\b v}\|_{\widehat K}\preceq h_K^{-1/2} 
|\widehat{\b v}|_{1, \widehat K}. 
\]
An estimate of the term on the right hand side follows in the same way as in the $H(\mathrm{curl})$ case:
\begin{align*}
  |\widehat{\nabla}\widehat{\b v}|   \preceq&\; 
  |J_KF_K^{-1}|_{1,\infty} |\b v|+|J_KF_K^{-1}|_{0,\infty}|F_K|_{0,\infty} |\grads \b v|
  \preceq \;h_K^3(| \b v|+|\nabla\b v|),
\end{align*}
so that
\[
\|\b v- \b\varPi_V \b v\|_{K}\preceq
h_K\|\b v\|_{1,K}.
\]

Finally, we  prove the $L^2$-estimate for $W_{h,0}$.
For a given function $ v\in H^1(K)$ on $K$, we associate the function 
 $\widehat{\b v}\equiv J_K v\circ \Phi_K \in H^1(\widehat K)$. 
By Theorem 6(c) part (4) and Remark 1, we have
\[
 \|\widehat {v}-\widehat { \varPi}_W\widehat{  v}\|_{\widehat K}\preceq 
  |\widehat{ v}|_{1, \widehat K}.
\]
 Hence, observing the relation $\varPi_W  v(\b x) = J_K^{-1}\widehat { \varPi}_W \widehat{ v}(\widehat{\b x})$ and using the bound 
for the determinant $J_K$ in \eqref{eq-1}, it follows that 
\[
\|v- \varPi_W v\|_{K}\preceq h_K^{-3/2} 
\|\widehat{v}- \widehat{ \varPi}_W\widehat{ v}\|_{\widehat K}\preceq h_K^{-3/2} 
|\widehat{ v}|_{1, \widehat K}. 
\]
To estimate the term on the right hand side, we note that 
\begin{align*}
  |\widehat{\nabla}\widehat{ v}|   \preceq&\; 
  |J_K|_{1,\infty} | v|+|J_K|_{0,\infty}|F_K|_{0,\infty} |\grads  v|
  \preceq \;h_K^4(|  v|+|\nabla v|),
\end{align*}
and hence
\[
\|v- \varPi_W v\|_{K}\preceq
h_K\|v\|_{1,K}.
\]

It remains to prove the estimates (1)--(4) for the derivatives.
In fact, these estimates are immediate consequence of 
the commuting diagram property. In particular, 
observing $\grads \varPi_S = \varPi_E \grads$, there holds
\begin{align*}
 \|
 \grads (v-\varPi_S v)
 \|_K = 
\|  \grads v-\b \varPi_E \grads v
 \|_K
 \preceq &\;h_K \|\grads v\|_{1,K},
\end{align*}
observing $\curls \b\varPi_E =\b \varPi_V \curls$, there holds
\begin{align*}
 \|
 \curls (\b v-\b\varPi_E\b v)
 \|_K = 
\|  \curls\b v-\b \varPi_V (\curls\b v)
 \|_K
 \preceq &\;h_K \|\curls \b v\|_{1,K},
\end{align*}
while
observing $\divs \b\varPi_V =\varPi_W \divs$, there holds
\begin{align*}
 \|
 \divs (\b v-\b\varPi_V\b v)
 \|_K = 
\|  \divs\b v-\varPi_W (\divs\b v)
 \|_K
 \preceq &\;h_K \|\divs \b v\|_{1,K}.
\end{align*}
\end{proof}

\section{Conclusion}
In this paper, we introduced a set of lowest-order composite finite elements for the de Rham complex on  pyramids.
The finite elements are compatible 
with those for the lowest-order {Raviart-Thomas-\NED}sequence on tetrahedra and hexahedra,
and, as such, can be used to construct conforming 
finite element spaces on {\sf THP} partitions consisting of tetrahedra, pyramids, and hexahedra.
Moreover, the finite element spaces deliver optimal error estimates on a sequence of successively refined {\sf THP}
partitions.


\bibliographystyle{siam}

\end{document}